\documentclass[10pt]{article}
\usepackage{epsfig}
\usepackage[T1]{fontenc}
\usepackage[utf8]{inputenc}
\usepackage[english]{babel}
\usepackage{color}
\usepackage{amsmath,enumerate}
\usepackage{amssymb}
\usepackage{amscd}
\usepackage{latexsym}
\usepackage{tabularx}
\usepackage{stmaryrd}
\usepackage{vmargin}

\usepackage{xcolor}

\usepackage{a4}
\setmarginsrb{3.5cm}{2cm}{3.5cm}{2cm}{1cm}{1cm}{1cm}{1cm}

\author{M. Montassier and P. Ochem\thanks{LIRMM (Universit\'e Montpellier 2, CNRS), Montpellier, France.}}

\title{Near-colorings: non-colorable graphs and NP-completeness}

\newenvironment{proof}{\par \noindent \textsc{Proof.} }{\hfill$\Box$\medskip}

\newtheorem{theorem}{Theorem}

\newtheorem{lemma}[theorem]{Lemma}

\newtheorem{claim}[theorem]{Claim}

\newtheorem{question}[theorem]{Question}

\begin{document}

\maketitle

\begin{abstract}
A graph $G$ is \emph{$(d_1,\ldots,d_l)$-colorable} if the vertex set of
$G$ can be partitioned into subsets $V_1,\ldots,V_l$ such that the
graph $G[V_i]$ induced by the vertices of $V_i$ has maximum degree at
most $d_i$ for all $1 \leq i\leq l$. In this paper, we focus on
complexity aspects of such colorings when $l=2,3$. More precisely, we
prove that, for any fixed integers $k,j,g$ with $(k,j)\ne (0,0)$ and
$g\ge 3$, either every planar graph with girth at least $g$ is
$(k,j)$-colorable or it is NP-complete to determine whether a planar
graph with girth at least $g$ is $(k,j)$-colorable. Also, for any
fixed integer $k$, it is NP-complete to determine whether a planar
graph that is either (0,0,0)-colorable or non-$(k,k,1)$-colorable is
$(0,0,0)$-colorable. Additionally, we exhibit non-$(3,1)$-colorable
planar graphs with girth 5 and non-$(2,0)$-colorable planar graphs
with girth 7.
\end{abstract} 

\section{Introduction}
A graph $G$ is \emph{$(d_1,\ldots,d_k)$-colorable} if the vertex set of
$G$ can be partitioned into subsets $V_1,\ldots,V_k$ such that the
graph $G[V_i]$ induced by the vertices of $V_i$ has maximum degree at
most $d_i$ for all $1 \leq i\leq k$. This notion generalizes those of
proper $k$-coloring (when $d_1=\ldots=d_k=0)$ and $d$-improper
$k$-coloring (when $d_1=\ldots=d_k=d\ge1)$. 

Planar graphs are known to be $(0,0,0,0)$-colorable (Appel and Haken
\cite{AHa77,AHb77}) and $(2,2,2)$-colorable (Cowen \emph{et
 al.}~\cite{CCW86}). Note that the result of Cowen \emph{et al.} is
optimal (for any integer $k$, there exist non-$(k,k,1)$-colorable
planar graphs) and holds in the choosability case (Eaton and Hull
\cite{EH99} or \v{S}krekovski~\cite{Skr99}). This last result was then
improved for planar graphs with large girth or for graphs with low
maximum average degree. We recall that the girth of a graph $G$,
denoted by $g(G)$, is the length of a shortest cycle in $G$, and the
maximum average degree of a graph $G$, denoted by ${\rm mad(G)}$, is
the maximum of the average degrees of all subgraphs of $G$, i.e. ${\rm
 mad}(G)=\max\,\left\{2|E(H)| / |V(H)|\,, H\subseteq G\right\}$.

\paragraph{$(1,0)$-coloring.}
Glebov and Zambalaeva~\cite{GZ07} proved that every planar graph $G$
with girth at least 16 is $(1,0)$-colorable. This was then
strengthened by Borodin and Ivanova~\cite{BI09a} who proved that every
graph $G$ with ${\rm mad}(G)<\frac {7}3$ is $(1,0)$-colorable.  This
implies that every planar graph $G$ with girth at least 14 is
$(1,0)$-colorable. Borodin and Kostochka~\cite{BK11} then proved that
every graph $G$ with ${\rm mad}(G)\le\frac {12}{5}$ is $(1,0)$-colorable.
In particular, it follows that every planar graph
$G$ with girth at least 12 is $(1,0)$-colorable. On the other hand,
they constructed graphs $G$ with ${\rm mad}(G)$ arbitrarily close
(from above) to $\frac {12}{5}$ that are not $(1,0)$-colorable; hence
their upper bound on the maximum average degree is best possible.
Also, Esperet \emph{et al.}~\cite{EMOP11} constructed a
non-$(1,0)$-colorable planar graph with girth 9; hence planar graphs
with girth at least 10 or 11 may be $(1,0)$-colorable.
To our knowledge, the question is still open.

\paragraph{$(k,0)$-coloring with $k\geq 2$.}
Borodin \emph{et al.}~\cite{BIMR09} proved that every graph $G$ with
${\rm mad}(G) < \frac {3k+4}{k+2}$ is $(k,0)$-colorable. The proof
in~\cite{BIMR09} extends that in~\cite{BI09a} but does not work for $k=1$.
Moreover, they exhibited a non-$(k,0)$-colorable planar graph with girth 6.
Finally, Borodin and Kostochka~\cite{BK2011} proved that for $k\ge 2$, every
graph $G$ with ${\rm mad}(G)\le \frac{3k+2}{k+1}$ is $(k,0)$-colorable.
This result is tight in terms of maximum average degree.

\paragraph{$(k,1)$-coloring.} 
Recently, Borodin, Kostochka, and Yancey~\cite{BKY2011} proved that
every graph with ${\rm mad}(G)\le \frac{14}{5}$ is $(1,1)$-colorable,
and the restriction on ${\rm mad}(G)$ is sharp. In~\cite{BIMR09c}, it
is proven that every graph $G$ with ${\rm mad}(G)<
\frac{10k+22}{3k+9}$ is $(k,1)$-colorable for $k\ge 2$.

\paragraph{$(k,j)$-coloring.}
A first step was made by Havet and Sereni~\cite{HS06} who showed that,
for every $k\ge 0$, every graph $G$ with ${\rm mad}(G)<\frac{4k+4}{k+2}$
is $(k,k)$-colorable (in fact $(k,k)$-choosable). More generally, they studied $k$-improper
$l$-choosability and proved that every graph $G$ with ${\rm mad (G)} <
l + \frac{lk}{l+k}$ ($l\ge 2, k\ge 0$) is $k$-improper $l$-choosable ;
this implies that such graphs are $(k,\ldots,k)$-colorable (where the
number of partite sets is $l$). Borodin \emph{et al.}~\cite{BIMR09d}
gave some sufficient conditions of $(k,j)$-colorability depending on
the density of the graphs using linear programming.
Finally, Borodin and Kostochka~\cite{BK2011} solved the
problem for a wide range of $j$ and $k$: let $j\ge 0$ and $k\ge 2j+2$;
every graph $G$ with ${\rm mad}(G)\le 2(2-\frac{k+2}{(j+2)(k+1)})$
is $(k,j)$-colorable. This result is tight in terms of the maximum average
degree and improves some results in~\cite{BIMR09,BIMR09c,BIMR09d}.

\medskip
Using the fact that every planar graph $G$ with girth $g(G)$ has ${\rm
 mad}(G)<2g(G)/(g(G)-2)$, the previous results give results for
planar graphs. They are summarized in Table~\ref{nearColoring-tab01}.

\begin{table}[htbp]
\label{nearColoring-tab01}
\begin{center}
\begin{tabular}{|c|c|c|c|c|c|}
\hline
girth & $(k,0)$ & $(k,1)$ & $(k,2)$ & $(k,3)$ & $(k,4)$ \\ \hline
3,4 & $\times$ & $\times$ & $\times$ & $\times$ & $\times$ \\ \hline
5 & $\times$ & ? & $(6,2)$~\cite{BK2011}& ? & $(4,4)$~\cite{HS06}\\ \hline
6 & $\times$~\cite{BIMR09}& $(4,1) $~\cite{BK2011} & $(2,2)$~\cite{HS06}& & \\ \hline
7 & $(4,0)$~\cite{BK2011}& $(1,1)$~\cite{BKY2011}& & &\\ \hline
8 & $(2,0)$~\cite{BK2011}& & & &\\ \hline
12 & $(1,0)$~\cite{BK11}& & & &\\ \hline
\end{tabular}
\caption{The girth and the $(k,j)$-colorability of planar graphs.
 The symbol ``$\times$'' means that there exist
 non-$(k,j)$-colorable planar graphs for any value of $k$.
 %Q3: Does there exist $k$ such that every planar graph with girth 5 is $(k,1)$-colorable?
 %Q6: Find the smallest value $k$ such that every planar graph with girth 5 is $(k,3)$-colorable? ($k\le 5$).
 }
\end{center}
\end{table}

From the previous discussion, the following questions are natural: 

\begin{question}{\ }
\begin{enumerate}
 \item Are planar graphs with girth 10 $(1,0)$-colorable?
 \item Are planar graphs with girth 7 $(3,0)$-colorable?
 \item Are planar graphs with girth 6 $(1,1)$-colorable?
 \item Are planar graphs with girth 5 $(4,1)$-colorable or $(k,1)$-colorable for some $k$ ?
 \item Are planar graphs with girth 5 $(2,2)$-colorable?
\end{enumerate}
\end{question}

% \begin{question}
% Does there exist a non-$(1,0)$-colorable planar graph with girth 10 or 11 ?
% \end{question}
% 
% \begin{question}
% Does there exist a non-$(3,0)$-colorable planar graph with girth 7 ?
% \end{question}
% 
% \begin{question}
% Does there exist an integer $k$ such that every planar graph with girth 5 is $(k,1)$-colorable ?
% \end{question}
% 
% \begin{question}
% Exhibit a non-$(k,1)$-colorable planar graph with girth 6 with $k\in\{1,2,3\}$.
% Prove that every planar graph with girth 6 is
% $(k,1)$-colorable for some value $k\in \{1,2,3\}$.
% \end{question}
% 
% \begin{question}
% Exhibit a non-$(2,2)$-colorable planar graph with 5. 
% \end{question}
% 
% \begin{question}
% Find the smallest value $k$ such that every planar graph with girth 5
% is $(k,3)$-colorable? ($k\le 5$).
% \end{question}

\paragraph{$(d_1,\ldots,d_k)$-coloring.} Finally we would like to
mention two studies. Chang \emph{et al.}~\cite{CH+11} gave some
approximation results to Steinberg's Conjecture using
$(k,j,i)$-colorings. Dorbec {\em et al.}\cite{DK+12} studied the particular case of $(d_1,\ldots,d_k)$-coloring
where the value of each $d_i$ ($i=1...k$) is either 0 or some value $d$,
making the link between $(d,0)$-coloring~\cite{BK2011} and $(d,...,d)$-coloring~\cite{HS06}.

\bigskip
The aim of this paper is to provide complexity results on the subject
and to obtain non-colorable planar graphs showing that some above-mentioned results are optimal.

\begin{claim} \label{claim} There exist:
\begin{enumerate}
\item a non-$(k,j)$-colorable planar graph with girth 4, for every $k,j\ge 0$,
\item a non-$(3,1)$-colorable planar graph with girth 5,
\item a non-$(2,0)$-colorable planar graph with girth 7.
\end{enumerate}
\end{claim}

Claim~\ref{claim}.3 shows that the $(2,0)$-colorability of all planar
graphs with girth at least 8~\cite{BK2011} is a tight result.

\begin{theorem}{\ }
\label{2c}
Let $k$, $j$, and $g$ be fixed integers such that $(k,j)\ne(0,0)$ and $g\ge 3$.
Either every planar graph with girth at least $g$ is $(k,j)$-colorable
or it is NP-complete to determine whether a planar graph with girth at least $g$ is $(k,j)$-colorable.
\end{theorem}

\begin{theorem}{\ }
\label{3c}
Let $k$ be a fixed integer.
It is NP-complete to determine whether a planar graph that is either $(0,0,0)$-colorable or non-$(k,k,1)$-colorable is $(0,0,0)$-colorable.
\end{theorem}

\bigskip
We construct
a non-$(k,j)$-colorable planar graph with girth 4 in Section 2,
a non-$(3,1)$-colorable planar graph with girth 5 in Section 3,
and a non-$(2,0)$-colorable planar graph with girth 7 in Section 4.
We prove Theorem~\ref{2c} in Section 5 and we prove Theorem~\ref{3c} in Section 6.

% 
% \begin{center}
% \begin{tabular}{|c|c|c|c|c|c|}
% \hline
% $k,j$ & 0 & 1 & 2 & 3 & 4 \\ \hline
% 1 & $9\le g_{1,0}<12$~\cite{BK11,EMOP11}& $5\le g_{1,1} < 7 $ Q4 & \cellcolor{gray} & \cellcolor{gray} & \cellcolor{gray} \\ \hline
% 2 & $g_{2,0}=7$~\cite{BK2011} + Section 4 & $5\le g_{2,1} < 7 $ Q4& $4\le g_{2,2} < 6 $ Q5 & \cellcolor{gray} & \cellcolor{gray} \\ \hline
% 3 & $6\le g_{3,0}<8$ Q2& $5\le g_{3,1} < 7 $ Q4 & $4\le g_{3,2} < 6 $b& $4\le g_{3,3} < 6 $ Q6 & \cellcolor{gray} \\ \hline
% 4 & $g_{4,0}=6$~\cite{BIMR09,BK2011}& $4\le g_{4,1} < 6 $ Q3 & $4\lebg_{4,2} < 6 $ & $4\le g_{4,3} < 6 $ Q6& $g_{4,4} = 4$\\ \hline
% 5 & $g_{5,0}=6$ & $4\le g_{5,1} < 6 $ Q3 & $4\le g_{5,2} < 6 $ & $4\le g_{5,3} < 6 $ Q6& $g_{5,4} = 4$\\ \hline
% 6 & $g_{6,0}=6$ & $4\le g_{6,1} < 6 $ Q3 & $g_{6,2} = 4$ & $g_{6,3} = 4$ & $g_{6,4} = 4$ \\ \hline
% 
% \end{tabular}
% \end{center}

\paragraph{Notations.} In the following, when we consider a
$(d_1,\ldots,d_k)$-coloring of a graph $G$, we color the vertices of
$V_i$ with color $d_i$ for $1\le i \le k$: for example in a
$(3,0)$-coloring, we will use color 3 to color the vertices of $V_1$
inducing a subgraph with maximum degree 3 and use color 0 to color the
vertices of $V_2$ inducing a stable set. A vertex is said to be
\emph{colored $i^j$} if it is colored $i$ and has $j$ neighbors colored $i$,
that is, it has degree $j$ in the subgraph induced by its color.
A vertex is \emph{saturated} if it is colored $i^i$, that is, if it has maximum degree
in the subgraph induced by its color.
A cycle (resp. face) of length $k$ is called a $k$-cycle (resp. $k$-face).
A $k$-vertex (resp. $k^-$-vertex, $k^+$-vertex) is a vertex of degree $k$ (resp. at
most $k$, at least $k$). The minimum degree of a graph $G$ is denoted by $\delta(G)$.

\section{A non-$(k,j)$-colorable planar graph with girth 4}
\label{g4}

We construct a non-$(k,j)$-colorable planar graph $G$ with girth 4,
with $k\ge j\ge 0$. Let $H_{x,y}$ be a copy of $K_{2,k+j+1}$, as depicted in Figure~\ref{figkj}.
In any $(k,j)$-coloring of $H_{x,y}$, the vertices $x$ and $y$ must receive the same color.
We obtain $G$ from a vertex $u$ and a star $S$ on $k+2$ vertices $v_1,\ldots,v_{k+2}$
(where $v_1$ is the center of $S$) by linking $u$ to each vertex $v_i$ with a copy $H_{u,v_i}$ of $H_{x,y}$.
The graph $G$ is not $(k,j)$-colorable: by the property of $H_{x,y}$, all the vertices $v_i$
should get the same color as $u$. This gives a monochromatic $S$, which is forbidden.
Notice that $G$ is $K_4$-minor free and 2-degenerate.
\begin{figure}[htbp]
\begin{center}
\includegraphics[width=10cm]{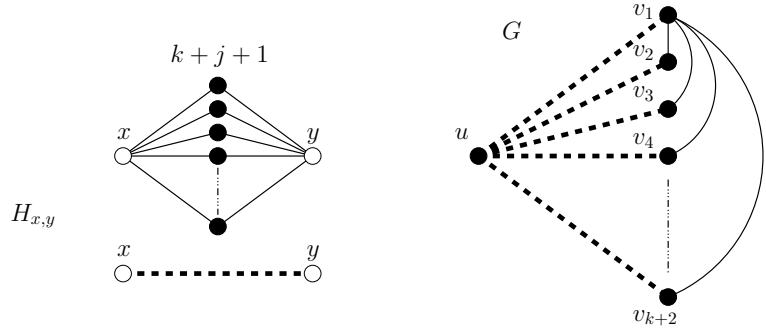}
\caption{A non-$(k,j)$-colorable $K_4$-minor free graph with girth 4.}
\label{figkj}
\end{center}
\end{figure}

\section{A non-$(3,1)$-colorable planar graph with girth 5}
\label{g5}

We construct a non-$(3,1)$-colorable planar graph with girth 5.
Consider the graph $H_{x,y}$ depicted in Figure~\ref{fig31}.
If $x$ and $y$ are colored 3 but have no neighbor colored 3, then it is not
possible to extend this partial coloring to $H_{x,y}$. Now, we
construct the graph $S_{z,r,s,t}$ as follows. Let $z$ be a vertex and
$rst$ be a path on three vertices. Take seven copies $H_{x_1,y_1},\ldots,H_{x_7,y_7}$
of $H_{x,y}$ and identify all $x_i$ ($i=1...7$)
with $z$ and all $y_i$ ($i=1...7$) with $r$. Repeat this construction
between $z$ and $s$, and between $z$ and $t$. Finally we obtain $G$ by
taking three copies of $S_{z,r,s,t}$, say
$S_{z_1,r_1,s_1,t_1},S_{z_2,r_2,s_2,t_2},S_{z_3,r_3,s_3,t_3}$, and by
adding an edge between $z_1$ and $z_2$, and $z_2$ and $z_3$ (Figure~\ref{fig31}).
The obtained graph is planar and has girth 5. Moreover,
it is not $(3,1)$-colorable. To see this, observe that a path on three
vertices must contain a vertex colored 3. Hence one of
$z_1,z_2,z_3$ is colored 3, as well, one $r_1,s_1,t_1$
(resp. $r_2,s_2,t_2$, $r_3,s_3,t_3$) must be colored 3. Suppose,
for example, that we have $z_1$ and $t_1$ colored 3. Between
these two vertices we have seven copies of $H_{x,y}$. It follows that
in one of these copies, $z_1$ and $t_1$ have no neighbors colored 3.
Consequently, this copy of $H_{x,y}$ is not $(3,1)$-colorable, and thus
$G$ is not $(3,1)$-colorable.

\begin{figure}[htbp]
\begin{center}
\includegraphics[width=8cm]{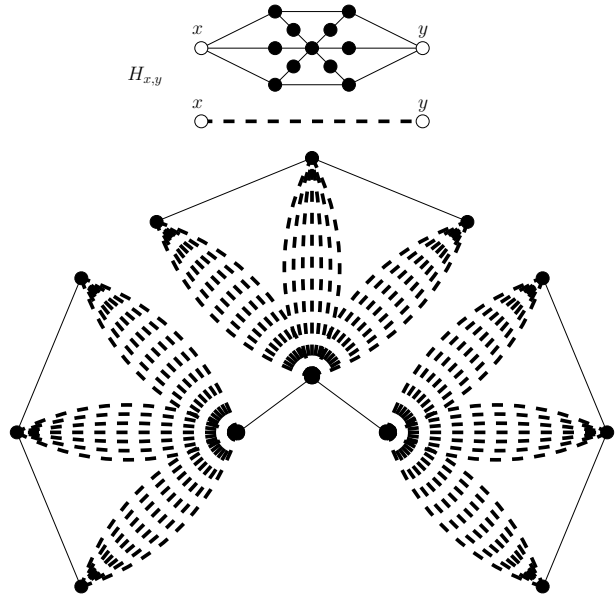}
\caption{A non-$(3,1)$-colorable 2-outerplanar graph with girth 5.}
\label{fig31}
\end{center}
\end{figure}

%This gives a partial answer to Question 3: if $k$ exists, then $k$ is at least 4.
Notice that $G$ is 2-outerplanar. Moreover, $G$ is 2-degenerate, this fact will be useful for the proof of Theorem~\ref{2c}.

\section{A non-$(2,0)$-colorable planar graph with girth 7}
\label{g7}

We give the construction of a non-$(2,0)$-colorable planar graph with girth 7.
Consider the graph $T_{x,y,z}$ in Figure~\ref{fig20-1}.
If the vertices $x$, $y$, and $z$ are colored 2 and have no neighbor colored 2, then $w$ is colored $2^2$.
Consider now the graph $S$ in Figure~\ref{fig20-1}.
Suppose that $a,b,c,d,e,f,g$ are respectively colored 2, 0, 2, 2, 2, 2, 0,
and that $a$ has no neighbor colored 2. Using successively the property of $T_{x,y,z}$,
we have that $w_1$, $w_2$, and $w_3$ must be colored $2^2$. It follows
that $w_4$ is colored 0, $w_5$ is colored 2, and so $w_6$ is colored $2^2$.
Again, by the property of $T_{x,y,z}$, $w_7$ must be colored $2^2$.
Finally, $w_8$ must be colored 0 and there is no choice of color for $w_9$.
Hence, such a coloring of the outer 7-cycle $abcdefg$ cannot be extended.

\begin{figure}[htbp]
\begin{center}
\includegraphics[width=8cm]{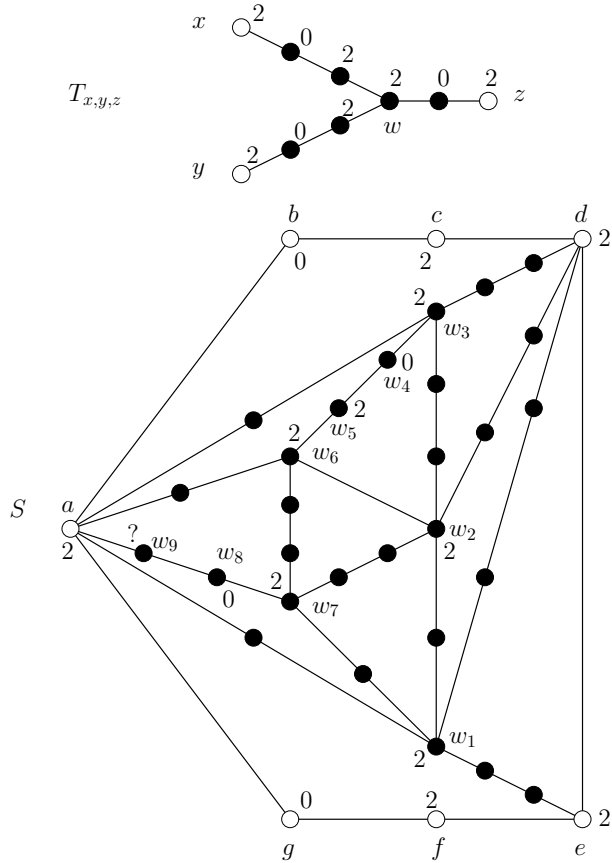}
\caption{Graphs $T_{x,y,z}$ and $S$.}
\label{fig20-1}
\end{center}
\end{figure}

The graph $H_z$ depicted in Figure~\ref{fig20-2} is
obtained from a vertex $z$ and a 7-cycle $v_1..v_7$ by linking $z$
to each $v_i$ ($i=1..7$) with a path on four vertices and by
embedding the graph $S$ in each face $F_i$ ($i=1..7$) (by identifying
the outer 7-cycle of $S$ with the 7-cycle bording the 7-face $F_i$).

Finally, the graph $G$ is obtained from a path on two vertices $uv$ and
six copies $H_{z_1},..,H_{z_6}$ of $H_z$ by identifying
$z_1,z_2,z_3$ with $u$ and $z_4,z_5,z_6$ with $v$. Observe that $G$
is planar and has girth 7. Let us prove that $G$ is not $(2,0)$-colorable:

\begin{enumerate}
\item $u$ and $v$ cannot be both colored 0, so without loss of generality, $u$ is colored 2.
\item In one of the three copies of $H_z$ attached to $u$, say $H_{z_1}$, $u$ has no
 neighbor colored 2.
\item Since a 7-cycle is not 2-colorable, the 7-cycle $v_1....v_7$ of
 $H_{z_1}$ has a monochromatic edge colored 2, say $v_1v_2$.
\item At last, the coloring of the face $F_2$ cannot be extended to the
 copy of $S$ embedded in $F_2$.
\end{enumerate}

\begin{figure}[htbp]
\begin{center}
\includegraphics[width=12cm]{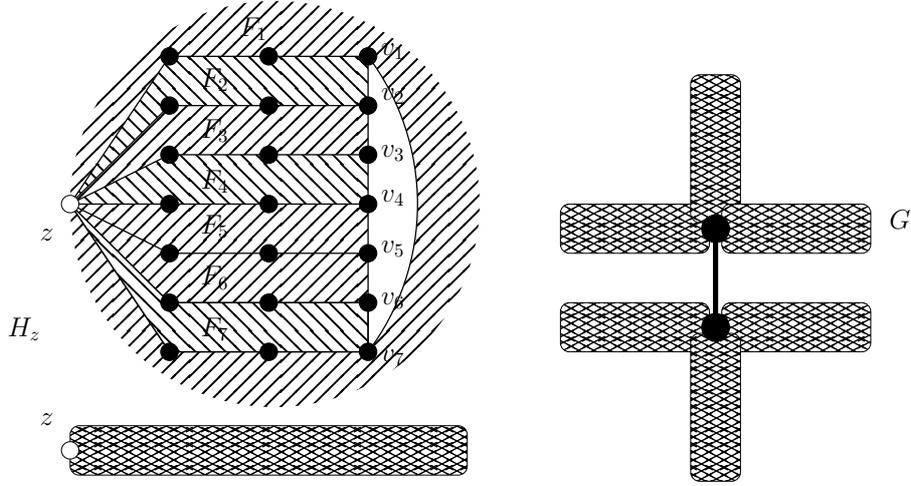}
\caption{Graphs $H_{z}$ and $D$.}
\label{fig20-2}
\end{center}
\end{figure}

\section{NP-completeness of $(k,j)$-colorings}

\begin{figure}[htbp]
\begin{center}
\includegraphics[width=6cm]{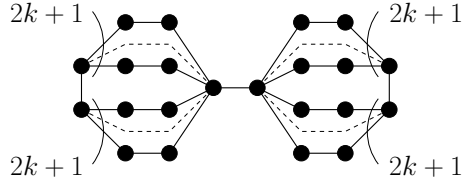}
\caption{A non-$(k,0)$-colorable planar graph with girth 6~\cite{BIMR09}.}
\label{figM6}
\end{center}
\end{figure}

Let $g_{k,j}$ be the largest integer $g$ such that there exists a planar graph with girth $g$ that is not $(k,j)$-colorable.
Because of large odd cycles, $g_{0,0}$ is not defined. For $(k,j)\ne(0,0)$, we have $4\le g_{k,j}\le11$ by the example in Figure~\ref{figkj}
and the result that planar graphs with girth at least 12 are $(0,1)$-colorable~\cite{BK11}. 
We prove this equivalent form of Theorem~\ref{2c}:
\begin{theorem}{\ }
\label{2c'}
Let $k$ and $j$ be fixed integers such that $(k,j)\ne(0,0)$.
It is NP-complete to determine whether a planar graph with girth $g_{k,j}$ is $(k,j)$-colorable.
\end{theorem}

Let us define the partial order $\preceq$. Let $n_3(G)$ be the number
of vertices of degree at least 3 in $G$. For any two graphs $G_1$ and $G_2$, we
have $G_1\prec G_2$ if and only if at least one of the following
conditions holds:
\begin{itemize}
\item $|V(G_1)|<|V(G_2)|$ and $n_3(G_1)\le n_3(G_2)$.
\item $n_3(G_1)<n_3(G_2)$.
\end{itemize} 
Note that the partial order $\preceq$ is well-defined
and is a partial linear extension of the subgraph poset.

The following lemma is useful.
% \begin{lemma}{\ }
% \label{2c0}
% Let $k\ge 2$ be a fixed integer.
% There exists a planar graph $G_{k,0}$ with girth $g_{k,0}$, minimally non-$(k,0)$-colorable for the subgraph order, such that $G_{k,0}$ contains two adjacent 2-vertices.
% \end{lemma}
% 
% \begin{proof}
% Consider a planar graph $G$ with girth $g_{k,0}$ that is minimally non-$(k,0)$-colorable for the order $\preceq$.
% By EXAMPLE MAILLE 6, $g_{k,0}\ge 6$, so $\delta(G)=2$ . Suppose that $G$ does not contain two adjacent 2-vertices.
% Then $G$ contains a 3-vertex adjacent to a 2-vertex, since otherwise we would have $MAD(G)\ge 3$.
% [config interdite 3-2 ].
% % By EXAMPLE MAILLE 7 and the result~\cite{BK2011} that planar graphs with girth 8 are $(2,0)$-colorable, we know that $g_{2,0}=7$.
% % Moreover, it is easy to check that every subgraph of EXAMPLE MAILLE 7 contains a vertex of degree at most one or two adjacent 2-vertices.
% % This proves the case $k=2$.
% % Similarly, by EXAMPLE MAILLE 6 and the result~\cite{BK2011} that planar graphs with girth 7 are $(4,0)$-colorable, we know that $g_{k,0}=6$ for $k\ge 4$.
% % Since every subgraph of EXAMPLE MAILLE 6 contains a vertex of degree at most one or two adjacent 2-vertices, we have proved the case $k\ge 4$.
% % Let $k\ge 3$ be a fixed integer
% \end{proof}
% 

\begin{lemma}{\ }
\label{2deg}
Let $k$ and $j$ be fixed integers such that $(k,j)\ne(0,0)$.
There exists a planar graph $G_{k,j}$ with girth $g_{k,j}$, minimally non-$(k,j)$-colorable for the subgraph order, such that $\delta(G_{k,j})=2$.
\end{lemma}

\begin{proof}
 We have $\delta(G_{k,j})\ge 2,$ since a non-$(k,j)$-colorable graph that is minimal for the subgraph order does not contain a vertex of degree at most 1.
 Notice that if for some pair $(k,j)$ we construct a 2-degenerate non-$(k,j)$-colorable planar graph with girth $g_{k,j}$,
 that is not necessarily minimal for the subgraph order,
 then the case of this pair $(k,j)$ is settled since some subgraph of such a graph must be minimal and have minimum degree 2.
 In particular, this proves the lemma for the following pairs $(k,j)$:
\begin{itemize}
 \item Pairs $(k,j)$ such that $g_{k,j}\le 4$: We actually have $g_{k,j}=4$ because of the example in Section~\ref{g4}, which is 2-degenerate.
 \item Pairs $(k,j)$ such that $g_{k,j}\ge 6$: Because a planar graph with girth at least 6 is 2-degenerate. In particular,
    the example in Figure~\ref{figM6} shows that $g_{k,0}\ge 6$, so the lemma is proved for all pairs $(k,0)$.
 \item Pairs $(k,1)$ such that $1\le k\le 3$: If $g_{k,j}\ge 6$, then we are in a previous case.
    Otherwise, we have $g_{k,j}=5$ because of the example in Section~\ref{g5}, and the lemma holds in this case since this example is 2-degenerate.
\end{itemize}

So all the remaining pairs satisfy $g_{k,j}=5$. There are two types of remaining pairs $(k,j)$:
\begin{itemize}
 \item Type 1: $k\ge 4$ and $j=1$.
 \item Type 2: $2\le j\le k$.
\end{itemize}

Consider a planar graph $G$ with girth 5 that is minimally non-$(k,j)$-colorable for the order $\preceq$ and suppose for contradiction
that $G$ does not contain a 2-vertex.
Also, suppose that $G$ contains a 3-vertex $a$ adjacent to three vertices $a_1$, $a_2$, and $a_3$ of degree at most 4.
For colorings of type 1, we can extend to $G$ a coloring of $G\setminus\{a\}$ by assigning to $a$ the color of improperty at least 4.
For colorings of type 2, we consider the graph $G'$ obtained from $G\setminus\{a\}$ by adding three 2-vertices $b_1$, $b_2$, and $b_3$ adjacent to,
respectively, $a_2$ and $a_3$, $a_1$ and $a_3$, $a_1$ and $a_2$.
Notice that $G'\preceq G$, so $G'$ admits a coloring $c$ of type 2. We can extend to $G$ the coloring of $G\setminus\{a\}$ induced by $c$ as follows:
If $a_1$, $a_2$, and $a_3$ are all assigned the same color, then we assign to $a$ the other color. Otherwise, we assign to $a$ the color that
appears at least twice among the colors of $b_1$, $b_2$, and $b_3$.
Now, since $G$ does not contain a 2-vertex nor a 3-vertex adjacent to three vertices of degree at most 4,
we have ${\rm mad(G)}\ge\frac{10}3$. This can be seen using the discharging procedure such that the initial charge of each vertex is its degree
and every vertex of degree at least 5 gives $\frac13$ to each adjacent 3-vertex.
The final charge of a 3-vertex is at least $3+\frac13=\frac{10}3$, the final charge of a 4-vertex is $4>\frac{10}3$, and the final charge of a $k$-vertex with $k\ge 5$ is
at least $k-k\times\frac13=\frac{2k}3\ge\frac{10}3$. Now, ${\rm mad(G)}\ge\frac{10}3$ contradicts the fact that $G$ is a planar graph with girth 5,
and this contradiction shows that $G$ contains a 2-vertex.
\end{proof}

We are now ready to prove Theorem~\ref{2c'}. The case of $(1,0)$-coloring is proved in~\cite{EMOP11} in a stronger form which takes
into account restrictions on both the girth and the maximum degree of the input planar graph.\\

Proof of the case $(k,0)$, $k\ge 2$.\\
We consider a graph $G_{k,0}$ as described in Lemma~\ref{2deg}, which contains a path $uxv$ where $x$ is a 2-vertex.
By minimality, any $(k,0)$-coloring of $G_{k,0}\setminus\{x\}$ is such that $u$ and $v$ get distinct saturated colors.
Let $G$ be the graph obtained from $G_{k,0}\setminus\{x\}$ by adding three 2-vertices $x_1$, $x_2$, and $x_3$ to create the path $ux_1x_2x_3v$.
So any $(k,0)$-coloring of $G$ is such that $x_2$ is colored $k^1$.
To prove the NP-completeness, we reduce the $(k,0)$-coloring problem to the $(1,0)$-coloring problem.
Let $I$ be a planar graph with girth $g_{1,0}$. For every vertex $s$ of $I$, add $(k-1)$ copies of $G$ such that the vertex
$x_2$ of each copy is adjacent to $s$, to obtain the graph $I'$.
By construction, $I'$ is $(k,0)$-colorable if and only if $I$ is $(1,0)$-colorable.
Moreover, $I'$ is planar, and since $g_{k,0}\le g_{1,0}$, the girth of $I'$ is $g_{k,0}$.\\

Proof of the case $(1,1)$.\\
There exist two independent proofs~\cite{FJLS03,FH03} that $(1,1)$-coloring is NP-complete for triangle-free planar graphs
with maximum degree 4. We use a reduction from that problem to prove that $(1,1)$-coloring is NP-complete for planar graphs
with girth $g_{1,1}$ (recall that $g_{1,1}$ is either 5 or 6 by the results in Section~ref{} and~\cite{BKY2011}).
We consider a graph $G_{1,1}$ as described in Lemma~\ref{2deg}, which contains a path $uxv$ where $x$ is a 2-vertex.
By minimality, any $(1,1)$-coloring of $G_{1,1}\setminus\{x\}$ is such that $u$ and $v$ get distinct saturated colors.
Let $G$ be the graph obtained from $G_{1,1}\setminus\{x\}$ by adding a vertex $u'$ adjacent to $u$ and a vertex $v'$ adjacent to $v$.
So any $(1,1)$-coloring of $G$ is such that $u'$ and $v'$ get distinct colors and $u'$ (resp. $v'$) has a color distinct from
the color of its (unique) neighbor. We construct the graph $E_{a,b}$ from two vertices $a$ and $b$ and two copies of $G$ such that
$a$ is adjacent to the vertices $u'$ of both copies of $G$ and $b$ is adjacent to the vertices $v'$ of both copies of $G$.
There exists a $(1,1)$-coloring of $E_{a,b}$ such that $a$ and $b$ have distinct colors and neither $a$ nor $b$ is saturated.
There exists a $(1,1)$-coloring of $E_{a,b}$ such that $a$ and $b$ have the same color.
Moreover, in every $(1,1)$-coloring of $E_{a,b}$ such that $a$ and $b$ have the same color, both $a$ and $b$ are saturated.

The reduction is as follows. Let $I$ be a planar graph.
For every edge $(p,q)$ of $I$, replace $(p,q)$ by a copy of $E_{a,b}$ such that $a$ is identified with $p$
and $b$ is identified with $q$, to obtain the graph $I'$.
By the properties of $E_{a,b}$, $I$ is $(1,1)$-colorable if and only if $I'$ is $(1,1)$-colorable.
Moreover, $I'$ is planar with girth $g_{1,1}$.
\\

Proof of the case $(k,j)$.\\
We consider a graph $G_{k,j}$ as described in Lemma~\ref{2deg}, which contains a path $uxv$ where $x$ is a 2-vertex.
By minimality, any $(k,j)$-coloring of $G_{k,j}\setminus\{x\}$ is such that $u$ and $v$ get distinct saturated colors.
Let $G$ be the graph obtained from $G_{k,j}\setminus\{x\}$ by adding a vertex $u'$ adjacent to $u$ and a vertex $v'$ adjacent to $v$.
So any $(k,j)$-coloring of $G$ is such that $u'$ and $v'$ get distinct colors and $u'$ (resp. $v'$) has a color distinct from
the color of its (unique) neighbor.
Let $t=\min(k-1,j)$. To prove the NP-completeness, we reduce the $(k,j)$-coloring to the $(k-t,j-t)$-coloring.
Thus the case $(k,k)$ reduces to the case $(1,1)$ which is NP-complete, and the case $(k,j)$ with $j<k$ reduces to the case $(k-j,0)$
which is NP-complete. The reduction is as follows.
%Let $I$ be a planar instance of $(k-t,j-t)$-coloring with girth $g_{k-t,k-t}$.
Let $I$ be a planar graph with girth $g_{k-t,j-t}$. For every vertex $s$ of $I$, add $t$ copies of $G$ such that the vertices
$u'$ and $v'$ of each copy is adjacent to $s$, to obtain the graph $I'$.
By construction, $I$ is $(k-t,j-t)$-colorable if and only if $I'$ is $(k,j)$-colorable.
Moreover, $I'$ is planar, and since $ g_{k,j}\le g_{k-t,j-t}$, the girth of $I'$ is $g_{k,j}$.

\section{NP-completeness of $(k,j,i)$-colorings}

In this section, we prove Theorem~\ref{3c} using a reduction from 3-colorability, i.e. $(0,0,0)$-colorability, which is NP-complete
for planar graphs~\cite{GJS76}.

Let $E$ be the graph depicted in Fig~\ref{fig:kk1}.
The graph $E'$ is obtained from $2k-1$ copies of $E$ by identifying the edge $ab$ of all copies.
Take 4 copies $E'_1$, $E'_2$, $E'_3$, and $E'_4$ of $E'$ and consider a triangle $T$ formed by the vertices
$y_0$, $x_0$, $x_1$ in one copy of $E$ in $E'_1$.
The graph $E''$ is obtained by identifying the edge $y_0x_0$ (resp. $y_0x_1$, $x_0x_1$) of $T$ with the edge $ab$ of $E'_2$ (resp. $E'_3$, $E'_4$). 
The edge $ab$ of $E'_1$ is then said to be the edge $ab$ of $E''$.

\begin{figure}[htbp]
\begin{center}
\includegraphics[width=12cm]{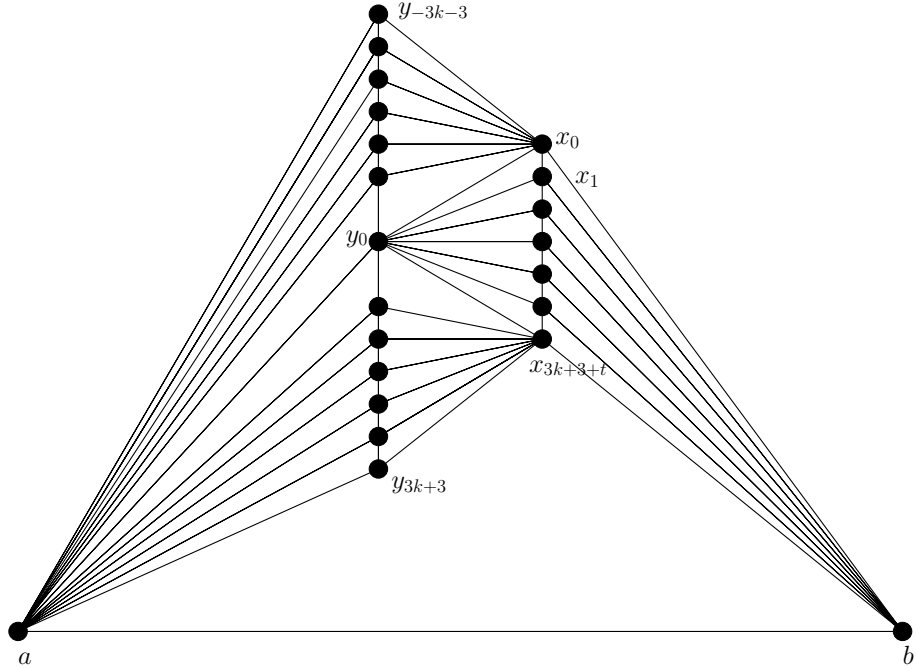}
\caption{The graph $E$. We take $t=0$ if $k$ is odd and $t=1$ if $k$ is even, so that $3k+3+t$ is even.}
\label{fig:kk1}
\end{center}
\end{figure}

\begin{lemma}{\ }
\label{lem:kk1}
\begin{enumerate}
 \item $E''$ admits a $(0,0,0)$-coloring.
 \item $E'$ does not admit a $(k,k,1)$-coloring such that $a$ and $b$ have a same color of improperty $k$.
 \item $E''$ does not admit a $(k,k,1)$-coloring such that $a$ and $b$ have the same color.
\end{enumerate}
\end{lemma}

\begin{proof}
\begin{enumerate}
 \item The following $(0,0,0)$-coloring $c$ of $E$ is unique up to permutation of colors:
 $c(a)=c(x_i)=1$ for even $i$, $c(b)=c(y_i)=2$ for even $i$, and $c(x_i)=c(y_i)=3$ for odd $i$.
 This coloring can be extended into a $(0,0,0)$-coloring of $E'$ and $E''$.
 \item Let $k_1$, $k_2$, and 1 denote the colors in a potential $(k,k,1)$-coloring $c$ of $E'$ such that $c(a)=c(b)=k_1$.
 By the pigeon-hole principle, one of the $2k-1$ copies of $E$ in $E'$, say $E^*$, is such that $a$ and $b$ are the only vertices with color $k_1$.
 So, one of the vertices $x_0$, $y_0$, and $x_{3k+3+t}$ in $E^*$ must get color $k_2$ since they cannot all get color $1$.
 We thus have a vertex $v_1\in\{a,b\}$ colored $k_1$ and vertex $v_2\in\{x_0, y_0, x_{3k+3+t}\}$ colored $k_2$ in $E^*$ which both dominate a path
 on at least $3k+3$ vertices. This path contains no vertex colored $k_1$ since it is in $E^*$. Moreover, it contains at most $k$ vertices colored $k_2$.
 On the other hand, every set of 3 consecutive vertices in this path contains at least one vertex colored $k_2$, so it contains at least $\frac{3k+3}3=k+1$
 vertices colored $k_2$. This contradiction shows that $E'$ does not admit a $(k,k,1)$-coloring such that $a$ and $b$ have a same color of improperty $k$.
 \item By the previous item and by construction of $E''$, if $E''$ admits a $(k,k,1)$-coloring $c$ such that $c(a)=c(b)$, then $c(a)=c(b)=1$.
 We thus have that $\{c(y_0),c(x_0),c(x_1)\}\subset\{k_1,k_2\}$. This implies that at least one edge of the triangle $T$ is monochromatic
 with a color of improperty $k$. By the previous item, the coloring $c$ cannot be extended to the copy of $E'$ attached to that monochromatic edge.
 This shows that $E''$ does not admit a $(k,k,1)$-coloring such that $a$ and $b$ have the same color.
\end{enumerate}
\end{proof}

We can now give a polynomial construction that transforms every planar graph $G$ into a planar graph $G'$ such that
$G'$ is $(0,0,0)$-colorable if $G$ is $(0,0,0)$-colorable and $G'$ is not $(k,k,1)$-colorable if $G$ is not $(0,0,0)$-colorable,
for every fixed integer $k$. 
The graph $G'$ is obtained from $G$ by identifying every edge of $G$ with the edge $ab$ of a copy of $E''$.
If $G$ is $(0,0,0)$-colorable, then this coloring can be extended into a $(0,0,0)$-coloring of $G'$ by Lemma~\ref{lem:kk1}.1.
On the other hand, if $G$ is not $(0,0,0)$-colorable, then every $(k,k,1)$-coloring $G$ contains a monochromatic edge $uv$, and then
the copy of $E''$ corresponding to $uv$ (and thus $G'$) does not admit a $(k,k,1)$-coloring by Lemma~\ref{lem:kk1}.3.

\end{document}